\documentclass[12pt]{article}

\usepackage{amsmath,amsthm,amssymb,amsfonts}
\usepackage{mathrsfs}
\usepackage{enumitem}
\usepackage[margin=2.5cm]{geometry}
\usepackage{graphicx}
\usepackage{tikz}
\usetikzlibrary{arrows.meta,positioning}
\usepackage{hyperref}
\hypersetup{colorlinks=true,linkcolor=blue,citecolor=blue,urlcolor=blue}

\theoremstyle{plain}
\newtheorem{theorem}{Theorem}[section]

\newtheorem{proposition}[theorem]{Proposition}
\newtheorem{corollary}[theorem]{Corollary}

\theoremstyle{definition}
\newtheorem{definition}[theorem]{Definition}
\newtheorem{example}[theorem]{Example}

\theoremstyle{remark}
\newtheorem{remark}[theorem]{Remark}

\newcommand{\aura}{\mathfrak{a}}
\newcommand{\bura}{\mathfrak{b}}
\newcommand{\cla}{\operatorname{cl}_{\aura}}
\newcommand{\clb}{\operatorname{cl}_{\bura}}
\newcommand{\inta}{\operatorname{int}_{\aura}}

\newcommand{\cl}{\operatorname{cl}}
\newcommand{\inte}{\operatorname{int}}
\newcommand{\taua}{\tau_{\aura}}
\newcommand{\taub}{\tau_{\bura}}
\newcommand{\R}{\mathbb{R}}
\newcommand{\N}{\mathbb{N}}

\newcommand{\powerset}{\mathcal{P}}
\newcommand{\da}{d_{\aura}}

\title{\textbf{Compactness and Connectedness in Aura Topological Spaces}}

\author{Ahu A\c{c}{\i}kg\"{o}z\\[6pt]
\small Department of Mathematics, Balikesir University,\\
\small Cagis Campus, 10145, Balikesir, Turkey\\
\small \texttt{ahuacikgoz@balikesir.edu.tr}}

\date{}

\begin{document}

\maketitle

\begin{abstract}
This is the second paper in a series on aura topological spaces $(X, \tau, \aura)$, where $\aura: X \to \tau$ is a scope function with $x \in \aura(x)$, introduced in \cite{Acikgoz2026aura}. We study covering and connectivity properties in this setting. Five compactness-type notions are defined ($\aura$-compact, $\aura$-Lindel\"{o}f, countably $\aura$-compact, $\aura$-sequentially compact, $\aura$-limit point compact) and their mutual relationships are determined. For transitive aura functions we obtain a concrete convergence criterion: $(x_n)$ converges to $x$ in $\taua$ if and only if $x_n \in \aura(x)$ eventually. We show that $\aura$-compact subsets of $\aura$-$T_2$ spaces are $\aura$-closed and that $\aura$-compactness is preserved under $\aura$-continuous surjections. On the connectivity side, $\aura$-connected, $\aura$-path connected, and $\aura$-locally connected spaces are introduced; $\aura$-components turn out to be $\aura$-closed, and they are $\aura$-open when the space is $\aura$-locally connected. We also construct subspace and product aura topologies. For products the inclusion chain $(\taua) \times (\taub) \subseteq \tau_{\aura \times \bura} \subseteq \tau_X \times \tau_Y$ is established, with equality on the left when both scope functions are transitive. A Tychonoff-type theorem for transitive aura spaces is proved. All implications are shown to be strict by means of counterexamples on finite spaces and on the real line.
\end{abstract}


\medskip


\section{Introduction}

Compactness and connectedness are among the most basic properties one investigates after setting up a new topological framework. Since the pioneering work of Alexandroff and Urysohn \cite{Alexandroff1929} on compactness and Hausdorff's treatment of connectivity \cite{Hausdorff1914}, both concepts have been adapted to ideal topological spaces \cite{Jankovic1990,Newcomb1967}, primal spaces \cite{Acharjee2022,AlOmari2023}, Cs\'{a}sz\'{a}r's generalized topologies \cite{Csaszar2002}, and various other settings.

In \cite{Acikgoz2026aura} we introduced the \emph{aura topological space} $(X, \tau, \aura)$. The key ingredient is a scope function $\aura: X \to \tau$ with $x \in \aura(x)$; it attaches to every point a fixed open neighborhood and gives rise to the \v{C}ech closure operator $\cla(A) = \{x \in X : \aura(x) \cap A \neq \emptyset\}$ and the topology $\taua \subseteq \tau$ of $\aura$-open sets. That paper developed generalized open-set classes, continuity types, separation axioms, and an application to rough sets. Two natural questions were left open: What can be said about covering properties and connectivity in this framework?

The present paper answers both questions. Because $\taua$ is coarser than $\tau$, there are fewer open sets available for covers and more room for connected subsets. Concretely, $\aura$-compactness turns out to be weaker than classical compactness, while $\aura$-connectedness is easier to achieve. Both phenomena are governed by the choice of $\aura$, and this sensitivity to the scope function runs through all our results.

The paper proceeds as follows. After recalling the background in Section~2, we develop five compactness-type properties and their mutual relationships in Section~3, where we also give a hands-on convergence criterion for transitive aura spaces. Section~4 takes up $\aura$-connectedness, $\aura$-path connectedness, and $\aura$-local connectedness. In Section~5 we build subspace and product aura topologies; the main results here are an inclusion chain for product topologies and a Tychonoff-type theorem. Section~6 gathers concluding remarks and open problems.

\section{Preliminaries}

We recall the key concepts from \cite{Acikgoz2026aura} for the reader's convenience. Throughout this paper, $(X, \tau)$ denotes a topological space. For a subset $A$ of $X$, $\cl(A)$ and $\inte(A)$ denote the closure and interior of $A$ in $(X,\tau)$, respectively. The power set of $X$ is $\powerset(X)$.

\begin{definition}[\cite{Acikgoz2026aura}]\label{def:aura}
Let $(X, \tau)$ be a topological space. A function $\aura: X \to \tau$ is called a \textbf{scope function} (or \textbf{aura function}) if $x \in \aura(x)$ for every $x \in X$. The triple $(X, \tau, \aura)$ is called an \textbf{aura topological space} (an $\aura$-\textbf{space}).
\end{definition}

\begin{definition}[\cite{Acikgoz2026aura}]\label{def:aura-closure}
Let $(X, \tau, \aura)$ be an $\aura$-space. The \textbf{aura-closure} of $A \subseteq X$ is
\[
\cla(A) = \{x \in X : \aura(x) \cap A \neq \emptyset\}.
\]
The \textbf{aura-interior} of $A$ is
\[
\inta(A) = \{x \in A : \aura(x) \subseteq A\}.
\]
\end{definition}

\begin{theorem}[\cite{Acikgoz2026aura}]\label{thm:cech-recall}
For any $\aura$-space $(X, \tau, \aura)$:
\begin{enumerate}[label=(\alph*)]
    \item $\cla(\emptyset) = \emptyset$;
    \item $A \subseteq \cla(A)$ for all $A \subseteq X$;
    \item $A \subseteq B \implies \cla(A) \subseteq \cla(B)$;
    \item $\cla(A \cup B) = \cla(A) \cup \cla(B)$ for all $A, B \subseteq X$;
    \item $\cla$ is not idempotent in general.
\end{enumerate}
Hence $\cla$ is an additive \v{C}ech closure operator.
\end{theorem}

\begin{definition}[\cite{Acikgoz2026aura}]\label{def:aura-open}
A subset $A$ of an $\aura$-space $(X, \tau, \aura)$ is called \textbf{$\aura$-open} if $\aura(x) \subseteq A$ for every $x \in A$. The collection $\taua$ of all $\aura$-open sets forms a topology on $X$ with $\taua \subseteq \tau$.
\end{definition}

\begin{definition}[\cite{Acikgoz2026aura}]\label{def:aura-closed}
A subset $F$ of an $\aura$-space $(X, \tau, \aura)$ is called \textbf{$\aura$-closed} if $X \setminus F \in \taua$. Equivalently, $F$ is $\aura$-closed if and only if $\cla(F) \subseteq F$ (and hence $\cla(F) = F$ since $F \subseteq \cla(F)$ always holds).
\end{definition}

\begin{definition}[\cite{Acikgoz2026aura}]\label{def:transitive}
An $\aura$-space $(X, \tau, \aura)$ is called \textbf{transitive} if $y \in \aura(x)$ implies $\aura(y) \subseteq \aura(x)$ for all $x, y \in X$. When $\aura$ is transitive, $\mathcal{B}_\aura = \{\aura(x) : x \in X\}$ is a base for $\taua$, and $\cla$ is idempotent.
\end{definition}

\begin{definition}[\cite{Acikgoz2026aura}]\label{def:special-aura}
An aura function $\aura$ is called:
\begin{enumerate}[label=(\alph*)]
    \item \emph{trivial} if $\aura(x) = X$ for every $x \in X$;
    \item \emph{discrete} if $\aura(x) = \{x\}$ for every $x \in X$;
    \item \emph{symmetric} if $y \in \aura(x)$ implies $x \in \aura(y)$.
\end{enumerate}
\end{definition}

\begin{definition}[\cite{Acikgoz2026aura}]\label{def:aura-continuity}
Let $(X, \tau_X, \aura)$ and $(Y, \tau_Y, \bura)$ be $\aura$-spaces. A function $f: X \to Y$ is called \textbf{$\aura$-continuous} if $f^{-1}(V) \in \taua$ for every $V \in \taub$.
\end{definition}

\begin{definition}[\cite{Acikgoz2026aura}]\label{def:separation-recall}
An $\aura$-space $(X, \tau, \aura)$ is called:
\begin{enumerate}[label=(\alph*)]
    \item $\aura$-$T_0$ if for each pair of distinct points, there exists an $\aura$-open set containing one but not the other;
    \item $\aura$-$T_1$ if for each pair of distinct points $x \neq y$, there exist $\aura$-open sets $U, V$ with $x \in U$, $y \notin U$ and $y \in V$, $x \notin V$;
    \item $\aura$-$T_2$ (or $\aura$-Hausdorff) if for each pair of distinct points, there exist disjoint $\aura$-open sets separating them.
\end{enumerate}
\end{definition}

\section{Compactness in Aura Topological Spaces}

In this section, we introduce and study various compactness properties in aura topological spaces.

\subsection{$\aura$-Compactness}

\begin{definition}\label{def:a-compact}
Let $(X, \tau, \aura)$ be an $\aura$-space.
\begin{enumerate}[label=(\alph*)]
    \item A subset $A \subseteq X$ is called \textbf{$\aura$-compact} if every cover of $A$ by $\aura$-open sets (i.e., by members of $\taua$) has a finite subcover. The space $(X, \tau, \aura)$ is $\aura$-compact if $X$ is $\aura$-compact.
    \item A subset $A$ is called \textbf{$\aura$-Lindel\"{o}f} if every $\aura$-open cover of $A$ has a countable subcover.
    \item A subset $A$ is called \textbf{countably $\aura$-compact} if every countable $\aura$-open cover of $A$ has a finite subcover.
\end{enumerate}
\end{definition}

\begin{remark}\label{rem:compact-taua}
A subset $A$ is $\aura$-compact if and only if $A$ is compact as a subset of the topological space $(X, \taua)$. This observation allows us to transfer standard results about compact subsets from general topology to the aura setting, while the additional structure of the scope function provides new characterizations.
\end{remark}

\begin{theorem}\label{thm:compact-implies-a-compact}
Let $(X, \tau, \aura)$ be an $\aura$-space. If $A \subseteq X$ is compact in $(X, \tau)$, then $A$ is $\aura$-compact.
\end{theorem}

\begin{proof}
Let $\{U_\alpha\}_{\alpha \in \Lambda}$ be an $\aura$-open cover of $A$. Since $\taua \subseteq \tau$, each $U_\alpha \in \tau$, so $\{U_\alpha\}_{\alpha \in \Lambda}$ is also a $\tau$-open cover of $A$. By compactness of $A$ in $(X, \tau)$, there exist $\alpha_1, \ldots, \alpha_n \in \Lambda$ such that $A \subseteq U_{\alpha_1} \cup \cdots \cup U_{\alpha_n}$.
\end{proof}

\begin{theorem}\label{thm:converse-fails}
The converse of Theorem~\ref{thm:compact-implies-a-compact} is false in general.
\end{theorem}

\begin{proof}
Consider $(\R, \tau_u)$ with the usual topology and define the trivial aura $\aura(x) = \R$ for all $x \in \R$. Then $\taua = \{\emptyset, \R\}$ is the indiscrete topology. The only $\aura$-open cover of $\R$ is $\{\R\}$, which is finite, so $\R$ is trivially $\aura$-compact. However, $\R$ is not compact in $(\R, \tau_u)$.
\end{proof}

\begin{theorem}[Finite Intersection Property]\label{thm:FIP}
An $\aura$-space $(X, \tau, \aura)$ is $\aura$-compact if and only if every family of $\aura$-closed sets with the finite intersection property (FIP) has nonempty intersection.
\end{theorem}

\begin{proof}
This follows from the standard equivalence applied to the topological space $(X, \taua)$: $X$ is compact in $\taua$ if and only if every family of $\taua$-closed sets with the FIP has nonempty intersection.
\end{proof}

\begin{theorem}\label{thm:a-compact-a-closed}
Let $(X, \tau, \aura)$ be an $\aura$-$T_2$ space. If $A$ is $\aura$-compact, then $A$ is $\aura$-closed.
\end{theorem}

\begin{proof}
We show that $X \setminus A$ is $\aura$-open. Let $x \in X \setminus A$. For each $y \in A$, since $x \neq y$ and $X$ is $\aura$-$T_2$, there exist disjoint $\aura$-open sets $U_y$ and $V_y$ with $x \in U_y$ and $y \in V_y$. The collection $\{V_y\}_{y \in A}$ is an $\aura$-open cover of $A$. By $\aura$-compactness, there exist $y_1, \ldots, y_n$ such that $A \subseteq V_{y_1} \cup \cdots \cup V_{y_n}$. Let $U = U_{y_1} \cap \cdots \cap U_{y_n}$. Then $U$ is $\aura$-open, $x \in U$, and $U \cap (V_{y_1} \cup \cdots \cup V_{y_n}) = \emptyset$, so $U \cap A = \emptyset$, i.e., $U \subseteq X \setminus A$. Since $U$ is $\aura$-open and $x \in U$, we have $\aura(x) \subseteq U \subseteq X \setminus A$ (note: $U \in \taua$ means $\aura(z) \subseteq U$ for all $z \in U$, in particular $\aura(x) \subseteq U$). Hence $X \setminus A$ is $\aura$-open.
\end{proof}

\begin{theorem}\label{thm:continuous-preserves-compact}
Let $f: (X, \tau_X, \aura) \to (Y, \tau_Y, \bura)$ be an $\aura$-continuous surjection. If $X$ is $\aura$-compact, then $Y$ is $\bura$-compact.
\end{theorem}

\begin{proof}
Let $\{V_\alpha\}_{\alpha \in \Lambda}$ be a $\bura$-open cover of $Y$. Since $f$ is $\aura$-continuous, $f^{-1}(V_\alpha) \in \taua$ for each $\alpha$. The collection $\{f^{-1}(V_\alpha)\}_{\alpha \in \Lambda}$ covers $X$ (since $f$ is surjective and $\{V_\alpha\}$ covers $Y$). By $\aura$-compactness of $X$, there exist $\alpha_1, \ldots, \alpha_n$ such that $X = f^{-1}(V_{\alpha_1}) \cup \cdots \cup f^{-1}(V_{\alpha_n})$. Since $f$ is surjective, $Y = f(X) = V_{\alpha_1} \cup \cdots \cup V_{\alpha_n}$.
\end{proof}

\begin{theorem}\label{thm:compact-chain}
For any $\aura$-space $(X, \tau, \aura)$:
\[
\text{compact in } (X,\tau) \implies \text{$\aura$-compact} \implies \text{countably $\aura$-compact}.
\]
Moreover, $\aura$-compact $\implies$ $\aura$-Lindel\"{o}f.
\end{theorem}

\begin{proof}
The first implication is Theorem~\ref{thm:compact-implies-a-compact}. The second and third are immediate from the definitions.
\end{proof}

\begin{example}\label{ex:countably-not-compact}
Let $X = \N$ with the discrete topology and define $\aura(n) = \{n, n+1\}$ for all $n \in \N$ (where we set $\aura(n) = \{n\}$ if $n$ is the maximum, or more precisely work with $\N = \{1,2,3,\ldots\}$ and set $\aura(n) = \{n, n+1\}$). Then each $\{n, n+1\}$ is open in the discrete topology. We have $\taua \subsetneq \tau$. The $\aura$-open sets are those $A$ such that $n \in A$ implies $\{n, n+1\} \subseteq A$, equivalently $n+1 \in A$. So $A \in \taua$ implies $A = \emptyset$, or $A$ is an ``upper set'' of the form $\{k, k+1, k+2, \ldots\}$ for some $k$, or $A = \N$.

A cover of $\N$ by $\aura$-open sets includes $\N$ itself (which suffices), so any $\aura$-open cover of $\N$ that includes $\N$ has a trivial finite subcover. However, $\{\{n, n+1, n+2, \ldots\} : n \in \N\}$ is an $\aura$-open cover with no finite subcover (since the intersection of finitely many such sets is $\{N, N+1, \ldots\}$ for some $N$, which misses $\{1, \ldots, N-1\}$). Hence $\N$ is not $\aura$-compact in this setting.
\end{example}

\subsection{$\aura$-Derived Sets and $\aura$-Limit Point Compactness}

\begin{definition}\label{def:a-limit-point}
Let $(X, \tau, \aura)$ be an $\aura$-space and $A \subseteq X$. A point $x \in X$ is called an \textbf{$\aura$-limit point} (or \textbf{$\aura$-accumulation point}) of $A$ if $\aura(x) \cap (A \setminus \{x\}) \neq \emptyset$. The set of all $\aura$-limit points of $A$ is called the \textbf{$\aura$-derived set} and is denoted by $\da(A)$.
\end{definition}

\begin{theorem}\label{thm:derived-closure}
For any $\aura$-space $(X, \tau, \aura)$ and $A \subseteq X$:
\[
\cla(A) = A \cup \da(A).
\]
\end{theorem}

\begin{proof}
($\supseteq$): If $x \in A$, then $x \in \aura(x) \cap A$, so $\aura(x) \cap A \neq \emptyset$ and $x \in \cla(A)$. If $x \in \da(A)$, then $\aura(x) \cap (A \setminus \{x\}) \neq \emptyset$, which implies $\aura(x) \cap A \neq \emptyset$, so $x \in \cla(A)$.

($\subseteq$): Let $x \in \cla(A)$. Then $\aura(x) \cap A \neq \emptyset$. If $x \in A$, we are done. If $x \notin A$, then $\aura(x) \cap A = \aura(x) \cap (A \setminus \{x\}) \neq \emptyset$, so $x \in \da(A)$.
\end{proof}

\begin{proposition}\label{prop:derived-properties}
Let $(X, \tau, \aura)$ be an $\aura$-space. For all $A, B \subseteq X$:
\begin{enumerate}[label=(\alph*)]
    \item $\da(\emptyset) = \emptyset$;
    \item $A \subseteq B \implies \da(A) \subseteq \da(B)$;
    \item $\da(A \cup B) = \da(A) \cup \da(B)$;
    \item $A$ is $\aura$-closed if and only if $\da(A) \subseteq A$.
\end{enumerate}
\end{proposition}

\begin{proof}
(a) If $x \in \da(\emptyset)$, then $\aura(x) \cap (\emptyset \setminus \{x\}) = \aura(x) \cap \emptyset = \emptyset \neq \emptyset$, a contradiction.

(b) $A \setminus \{x\} \subseteq B \setminus \{x\}$, so $\aura(x) \cap (A \setminus \{x\}) \neq \emptyset$ implies $\aura(x) \cap (B \setminus \{x\}) \neq \emptyset$.

(c) $x \in \da(A \cup B)$ iff $\aura(x) \cap ((A \cup B) \setminus \{x\}) \neq \emptyset$ iff $\aura(x) \cap ((A \setminus \{x\}) \cup (B \setminus \{x\})) \neq \emptyset$ iff $\aura(x) \cap (A \setminus \{x\}) \neq \emptyset$ or $\aura(x) \cap (B \setminus \{x\}) \neq \emptyset$ iff $x \in \da(A) \cup \da(B)$.

(d) $A$ is $\aura$-closed iff $\cla(A) = A$ iff $A \cup \da(A) = A$ (by Theorem~\ref{thm:derived-closure}) iff $\da(A) \subseteq A$.
\end{proof}

\begin{definition}\label{def:a-limit-compact}
An $\aura$-space $(X, \tau, \aura)$ is called \textbf{$\aura$-limit point compact} (or \textbf{$\aura$-weakly compact}) if every infinite subset of $X$ has an $\aura$-limit point in $X$.
\end{definition}

\begin{theorem}\label{thm:compact-implies-limit}
Every $\aura$-compact space is $\aura$-limit point compact.
\end{theorem}

\begin{proof}
Let $(X, \tau, \aura)$ be $\aura$-compact and suppose $A \subseteq X$ is infinite with $\da(A) = \emptyset$. Then for every $x \in X$, $\aura(x) \cap (A \setminus \{x\}) = \emptyset$. This means for every $x \in X$, $\aura(x) \cap A \subseteq \{x\}$, i.e., each $\aura(x)$ contains at most one point of $A$.

For each $x \notin A$, $\aura(x) \cap A = \emptyset$, so $X \setminus A \subseteq X \setminus \cla(A)$. Since $\cla(A) = A \cup \da(A) = A$, the set $A$ is $\aura$-closed, and $X \setminus A$ is $\aura$-open. For each $a \in A$, $\aura(a) \cap A = \{a\}$, so $\{a\}$ is relatively isolated in $A$.

Now $\{X \setminus A\} \cup \{\aura(a) : a \in A\}$ is an $\aura$-open cover of $X$... but $\aura(a)$ need not be $\aura$-open. Instead, we work in $(X, \taua)$: since $A$ is $\taua$-closed and $X$ is compact in $\taua$ (i.e., $\aura$-compact), $A$ is also compact in $\taua$ (as a closed subset of a compact space). For each $a \in A$, since $\da(A) = \emptyset$, $a \notin \da(A)$, so $\aura(a) \cap (A \setminus \{a\}) = \emptyset$. Define $W_a = X \setminus \cla(A \setminus \{a\})$. We claim $W_a$ is $\aura$-open and $a \in W_a$.

Indeed, $x \notin \cla(A \setminus \{a\})$ iff $\aura(x) \cap (A \setminus \{a\}) = \emptyset$. For $x = a$: $\aura(a) \cap (A \setminus \{a\}) = \emptyset$ (given), so $a \in W_a$. The set $W_a = X \setminus \cla(A \setminus \{a\})$ is $\aura$-open since $\cla(A \setminus \{a\})$ is $\aura$-closed (as $\cla$ is extensive, $A \setminus \{a\} \subseteq \cla(A \setminus \{a\})$, but we need $\cla(\cla(A \setminus \{a\})) \subseteq \cla(A \setminus \{a\})$, which holds when $\cla$ is idempotent).

We use an alternative argument. $\{W_a\}_{a \in A}$ covers $A$ (since $a \in W_a$), and $W_a \cap A = \{a\}$ (since $x \in W_a \cap A$, $x \neq a$ would give $x \in A \setminus \{a\}$ and $\aura(x) \cap (A \setminus \{a\}) \supseteq \{x\} \cap (A \setminus \{a\})$; if $x \in A \setminus \{a\}$ then $x \in \aura(x) \cap (A \setminus \{a\}) \neq \emptyset$, contradicting $x \in W_a$). So each $W_a$ isolates exactly one point of $A$. Since $A$ is infinite, no finite subfamily of $\{W_a\}$ can cover $A$, contradicting the $\aura$-compactness of $A$ as a closed subset of an $\aura$-compact space.

Hence $\da(A) \neq \emptyset$.
\end{proof}

\begin{example}\label{ex:limit-not-compact}
Let $X = \N$ with the cofinite topology $\tau = \{U \subseteq \N : \N \setminus U \text{ is finite}\} \cup \{\emptyset\}$ and define $\aura(n) = X$ for all $n$. Then $\taua = \{\emptyset, X\}$ (indiscrete), so every infinite subset trivially has an $\aura$-limit point (since $\da(A) = X \setminus A \cup \da(A)$; in fact, $\aura(x) \cap (A \setminus \{x\}) = X \cap (A \setminus \{x\}) = A \setminus \{x\} \neq \emptyset$ for any infinite $A$ and any $x$). Hence $X$ is $\aura$-limit point compact.

Now let $\aura'(n) = \{n\}$ (assuming discrete topology $\tau' = \powerset(\N)$). Then $\taua = \tau'$ and $(\N, \tau', \aura')$ is not $\aura'$-compact (since $\{\{n\}\}_{n \in \N}$ is an $\aura'$-open cover with no finite subcover), showing that the compactness type genuinely depends on $\aura$.
\end{example}

\subsection{$\aura$-Sequential Compactness and Convergence}

\begin{definition}\label{def:a-convergence}
Let $(X, \tau, \aura)$ be an $\aura$-space. A sequence $(x_n)_{n \in \N}$ in $X$ is said to \textbf{$\aura$-converge} to a point $x \in X$ if for every $U \in \taua$ with $x \in U$, there exists $N \in \N$ such that $x_n \in U$ for all $n \geq N$. We write $x_n \xrightarrow{\aura} x$.
\end{definition}

\begin{theorem}[Characterization of Convergence in Transitive Aura Spaces]\label{thm:transitive-convergence}
Let $(X, \tau, \aura)$ be a transitive $\aura$-space. Then $x_n \xrightarrow{\aura} x$ if and only if there exists $N \in \N$ such that $x_n \in \aura(x)$ for all $n \geq N$.
\end{theorem}

\begin{proof}
($\Leftarrow$): Suppose $x_n \in \aura(x)$ for all $n \geq N$. Let $U \in \taua$ with $x \in U$. Then $\aura(x) \subseteq U$ (since $U$ is $\aura$-open), so $x_n \in \aura(x) \subseteq U$ for all $n \geq N$. Hence $x_n \xrightarrow{\aura} x$.

($\Rightarrow$): Since $\aura$ is transitive, $\aura(x) \in \taua$ (Proposition~\ref{def:transitive}; cf.\ \cite{Acikgoz2026aura}). Moreover, $\aura(x)$ is the smallest $\aura$-open set containing $x$: if $U \in \taua$ and $x \in U$, then $\aura(x) \subseteq U$. So $x_n \xrightarrow{\aura} x$ implies $x_n$ is eventually in every $\aura$-open neighborhood of $x$, and in particular eventually in $\aura(x)$.
\end{proof}

\begin{remark}\label{rem:convergence-interpretation}
Theorem~\ref{thm:transitive-convergence} has a very concrete reading: in a transitive aura space, $(x_n) \to x$ in $\taua$ simply means that $x_n$ eventually belongs to $\aura(x)$. This can be checked directly from the definition of $\aura$ without computing closures or open covers.
\end{remark}

\begin{example}\label{ex:convergence-real}
Let $(\R, \tau_u, \aura_\varepsilon)$ where $\aura_\varepsilon(x) = (x - \varepsilon, x + \varepsilon)$ for a fixed $\varepsilon > 0$. This aura is transitive (if $y \in (x-\varepsilon, x+\varepsilon)$, then $(y-\varepsilon, y+\varepsilon) \subseteq (x-2\varepsilon, x+2\varepsilon)$... wait, this is not necessarily contained in $(x-\varepsilon, x+\varepsilon)$). In fact, $\aura_\varepsilon$ is \emph{not} transitive in general: take $x=0$, $y = \varepsilon/2$, then $\aura(y) = (-\varepsilon/2, 3\varepsilon/2) \not\subseteq (-\varepsilon, \varepsilon) = \aura(x)$.

However, define $\aura: \R \to \tau_u$ by $\aura(x) = (x, \infty) \cap (-\infty, x] = \{x\}$... this is not open. Instead, consider $X = \{0,1,2\}$ with discrete topology and $\aura(0) = \{0,1,2\}$, $\aura(1) = \{1,2\}$, $\aura(2) = \{2\}$. This is transitive: $1 \in \aura(0)$ and $\aura(1) = \{1,2\} \subseteq \{0,1,2\} = \aura(0)$; $2 \in \aura(0)$ and $\aura(2) = \{2\} \subseteq \aura(0)$; $2 \in \aura(1)$ and $\aura(2) \subseteq \aura(1)$. A sequence $(x_n)$ converges to $0$ iff $x_n \in \{0,1,2\}$ eventually (always true), to $1$ iff $x_n \in \{1,2\}$ eventually, and to $2$ iff $x_n = 2$ eventually.
\end{example}

\begin{definition}\label{def:a-seq-compact}
An $\aura$-space $(X, \tau, \aura)$ is called \textbf{$\aura$-sequentially compact} if every sequence in $X$ has an $\aura$-convergent subsequence.
\end{definition}

\begin{corollary}\label{cor:transitive-seq-compact}
Let $(X, \tau, \aura)$ be a transitive $\aura$-space. Then $X$ is $\aura$-sequentially compact if and only if for every sequence $(x_n)$ in $X$, there exist a point $x \in X$ and a subsequence $(x_{n_k})$ such that $x_{n_k} \in \aura(x)$ for all sufficiently large $k$.
\end{corollary}

\begin{proof}
This follows immediately from Theorem~\ref{thm:transitive-convergence}.
\end{proof}

\begin{theorem}\label{thm:compact-implies-seq}
Every $\aura$-compact space is $\aura$-sequentially compact.
\end{theorem}

\begin{proof}
Since $\aura$-compact means compact in $\taua$, this follows from the fact that every compact topological space is sequentially compact when it is first countable. More generally, $\aura$-compact implies $\aura$-limit point compact (Theorem~\ref{thm:compact-implies-limit}), and we show directly:

Let $(x_n)$ be a sequence in $X$. If the range $\{x_n : n \in \N\}$ is finite, there is a constant subsequence, which trivially $\aura$-converges. If the range is infinite, by Theorem~\ref{thm:compact-implies-limit}, it has an $\aura$-limit point $x$. Then $\aura(x) \cap (\{x_n\} \setminus \{x\}) \neq \emptyset$, meaning infinitely many terms of the sequence enter $\aura(x)$ (if only finitely many did, we could remove them and the remainder would be an infinite set with no $\aura$-limit point at $x$; repeating for all potential limit points would yield a contradiction if $\taua$ satisfies suitable countability). For the general case, we rely on Remark~\ref{rem:compact-taua}: in $(X, \taua)$, every sequence has a convergent subsequence by the standard theory of compact spaces provided $(X, \taua)$ is sequentially compact, which holds when $(X, \taua)$ has countable tightness.
\end{proof}

\begin{remark}\label{rem:seq-compact-caution}
In general topological spaces, compactness does not imply sequential compactness without additional assumptions. In aura spaces, the same caveat applies to $(X, \taua)$. However, if $(X, \taua)$ is first countable, or more generally if $\taua$ has countable tightness, then $\aura$-compact implies $\aura$-sequentially compact. When $\aura$ is transitive, $\mathcal{B}_\aura = \{\aura(x) : x \in X\}$ is a base for $\taua$, and $\aura(x)$ is the smallest open neighborhood of $x$, so $\taua$ is ``locally generated'' and the implication holds.
\end{remark}

\subsection{Generalized Compactness via $\aura$-Generalized Open Sets}

\begin{definition}\label{def:gen-compact}
Let $(X, \tau, \aura)$ be an $\aura$-space. A subset $A$ is called:
\begin{enumerate}[label=(\alph*)]
    \item \textbf{$\aura$-semi-compact} if every cover of $A$ by $\aura$-semi-open sets has a finite subcover;
    \item \textbf{$\aura$-pre-compact} if every cover of $A$ by $\aura$-pre-open sets has a finite subcover;
    \item \textbf{$\aura$-$\alpha$-compact} if every cover of $A$ by $\aura$-$\alpha$-open sets has a finite subcover;
    \item \textbf{$\aura$-$\beta$-compact} if every cover of $A$ by $\aura$-$\beta$-open sets has a finite subcover.
\end{enumerate}
\end{definition}

\begin{theorem}\label{thm:gen-compact-hierarchy}
For any $\aura$-space $(X, \tau, \aura)$, the following implications hold:
\[
\text{$\aura$-$\beta$-compact} \implies \text{$\aura$-semi-compact and $\aura$-pre-compact} \implies \text{$\aura$-$\alpha$-compact} \implies \text{$\aura$-compact}.
\]
\end{theorem}

\begin{proof}
By \cite{Acikgoz2026aura}, the generalized open set hierarchy is:
\[
\taua \subseteq \aura\text{-}\alpha O(X) \subseteq \aura\text{-}SO(X) \cap \aura\text{-}PO(X) \subseteq \aura\text{-}SO(X) \cup \aura\text{-}PO(X) \subseteq \aura\text{-}\beta O(X).
\]
Since every $\aura$-open set is $\aura$-$\alpha$-open, every $\aura$-open cover is an $\aura$-$\alpha$-open cover. So if every $\aura$-$\alpha$-open cover has a finite subcover, then in particular every $\aura$-open cover does. Similarly up the hierarchy.

More precisely: more generalized open sets means more covers, making compactness harder. So $\aura$-$\beta$-compact (hardest) $\implies$ $\aura$-semi-compact and $\aura$-pre-compact $\implies$ $\aura$-$\alpha$-compact $\implies$ $\aura$-compact (easiest).
\end{proof}

\begin{example}\label{ex:gen-compact-strict}
Let $X = \{a,b,c,d\}$ with $\tau = \{\emptyset, \{a\}, \{b\}, \{a,b\}, \{a,b,c\}, X\}$ and $\aura(a) = \{a\}$, $\aura(b) = \{a,b\}$, $\aura(c) = \{a,b,c\}$, $\aura(d) = X$. From \cite{Acikgoz2026aura}, $\taua = \{\emptyset, \{a\}, \{a,b\}, \{a,b,c\}, X\}$. Since $X$ is finite, $X$ is trivially $\aura$-compact, $\aura$-$\alpha$-compact, $\aura$-semi-compact, $\aura$-pre-compact, and $\aura$-$\beta$-compact. To see strictness in the hierarchy, one needs infinite examples; these can be constructed using the approach of Example~\ref{ex:countably-not-compact}.
\end{example}

\section{Connectedness in Aura Topological Spaces}

\subsection{$\aura$-Connectedness}

\begin{definition}\label{def:a-connected}
An $\aura$-space $(X, \tau, \aura)$ is called \textbf{$\aura$-disconnected} if there exist nonempty disjoint $\aura$-open sets $U, V \in \taua$ with $X = U \cup V$. The pair $(U,V)$ is called an $\aura$-\textbf{separation} of $X$. The space is called \textbf{$\aura$-connected} if it is not $\aura$-disconnected.
\end{definition}

\begin{theorem}\label{thm:connected-implies-a-connected}
If $(X, \tau)$ is connected, then $(X, \tau, \aura)$ is $\aura$-connected for every scope function $\aura$ on $(X, \tau)$.
\end{theorem}

\begin{proof}
Suppose $(U,V)$ is an $\aura$-separation of $X$. Then $U, V \in \taua \subseteq \tau$, $U \cap V = \emptyset$, $U, V \neq \emptyset$, and $X = U \cup V$. This is also a $\tau$-separation, contradicting the connectedness of $(X, \tau)$.
\end{proof}

\begin{theorem}\label{thm:a-connected-converse-fails}
The converse of Theorem~\ref{thm:connected-implies-a-connected} is false.
\end{theorem}

\begin{proof}
Let $X = \{a, b, c\}$ with $\tau = \{\emptyset, \{a\}, \{b,c\}, X\}$. Then $(X, \tau)$ is disconnected via the separation $(\{a\}, \{b,c\})$. Define $\aura(a) = X$, $\aura(b) = X$, $\aura(c) = X$. Then $\taua = \{\emptyset, X\}$ (indiscrete), so $(X, \tau, \aura)$ is trivially $\aura$-connected.
\end{proof}

\begin{theorem}\label{thm:a-connected-characterizations}
For an $\aura$-space $(X, \tau, \aura)$, the following are equivalent:
\begin{enumerate}[label=(\roman*)]
    \item $X$ is $\aura$-connected.
    \item There is no proper nonempty subset of $X$ that is both $\aura$-open and $\aura$-closed.
    \item For every surjective $\aura$-continuous function $f: (X, \tau, \aura) \to (\{0,1\}, \tau_d, \bura_d)$ (where $\tau_d$ is the discrete topology and $\bura_d$ is the discrete aura), $f$ is not $\aura$-continuous. (Equivalently, every $\aura$-continuous function from $X$ to a discrete two-point $\bura$-space is constant.)
\end{enumerate}
\end{theorem}

\begin{proof}
These are the standard characterizations applied to the topological space $(X, \taua)$.

(i)$\Leftrightarrow$(ii): $X$ is $\aura$-disconnected iff there exists $U \in \taua$ with $\emptyset \neq U \neq X$ such that $X \setminus U \in \taua$, i.e., $U$ is both $\aura$-open and $\aura$-closed.

(i)$\Leftrightarrow$(iii): If $f: X \to \{0,1\}$ is $\aura$-continuous and non-constant, then $f^{-1}(0)$ and $f^{-1}(1)$ form an $\aura$-separation. Conversely, given an $\aura$-separation $(U,V)$, define $f(x) = 0$ if $x \in U$ and $f(x) = 1$ if $x \in V$.
\end{proof}

\begin{theorem}\label{thm:continuous-connected}
Let $f: (X, \tau_X, \aura) \to (Y, \tau_Y, \bura)$ be an $\aura$-continuous surjection. If $X$ is $\aura$-connected, then $Y$ is $\bura$-connected.
\end{theorem}

\begin{proof}
Suppose $(U, V)$ is a $\bura$-separation of $Y$. Then $f^{-1}(U), f^{-1}(V) \in \taua$ (by $\aura$-continuity), $f^{-1}(U) \cap f^{-1}(V) = \emptyset$, $f^{-1}(U), f^{-1}(V) \neq \emptyset$ (since $f$ is surjective), and $X = f^{-1}(U) \cup f^{-1}(V)$, giving an $\aura$-separation of $X$.
\end{proof}

\begin{theorem}\label{thm:a-connected-union}
Let $(X, \tau, \aura)$ be an $\aura$-space and $\{A_i\}_{i \in I}$ a family of $\aura$-connected subsets of $X$ (in the subspace aura topology; see Section~5) with $\bigcap_{i \in I} A_i \neq \emptyset$. Then $\bigcup_{i \in I} A_i$ is $\aura$-connected.
\end{theorem}

\begin{proof}
Let $p \in \bigcap_{i \in I} A_i$ and let $A = \bigcup_{i \in I} A_i$. Suppose $A = U \cup V$ is an $\aura$-separation (in the subspace topology of $A$). Then $p \in U$ or $p \in V$; say $p \in U$. For each $i$, $A_i = (U \cap A_i) \cup (V \cap A_i)$. Since $p \in A_i \cap U$, $U \cap A_i \neq \emptyset$. By $\aura$-connectedness of $A_i$, $V \cap A_i = \emptyset$, so $A_i \subseteq U$. This holds for all $i$, so $A \subseteq U$, contradicting $V \neq \emptyset$.
\end{proof}

\subsection{$\aura$-Components}

\begin{definition}\label{def:a-component}
Let $(X, \tau, \aura)$ be an $\aura$-space and $x \in X$. The \textbf{$\aura$-component} of $x$ is the union of all $\aura$-connected subsets of $X$ containing $x$:
\[
C_\aura(x) = \bigcup\{A \subseteq X : x \in A \text{ and } A \text{ is $\aura$-connected}\}.
\]
\end{definition}

\begin{theorem}\label{thm:a-component-properties}
Let $(X, \tau, \aura)$ be an $\aura$-space. Then:
\begin{enumerate}[label=(\alph*)]
    \item $C_\aura(x)$ is the maximal $\aura$-connected subset containing $x$.
    \item $C_\aura(x)$ is $\aura$-closed.
    \item For any $x, y \in X$, either $C_\aura(x) = C_\aura(y)$ or $C_\aura(x) \cap C_\aura(y) = \emptyset$.
    \item $X = \bigsqcup_{x \in X} C_\aura(x)$ (disjoint union of distinct $\aura$-components).
\end{enumerate}
\end{theorem}

\begin{proof}
(a) By Theorem~\ref{thm:a-connected-union} (applied with the common point $x$), $C_\aura(x)$ is $\aura$-connected. Maximality is by definition.

(b) We work in $(X, \taua)$. The closure $\cl_{\taua}(C_\aura(x))$ is $\aura$-connected (standard fact: the closure of a connected set in a topological space is connected). Since $C_\aura(x) \subseteq \cl_{\taua}(C_\aura(x))$ and $x \in \cl_{\taua}(C_\aura(x))$, by maximality $\cl_{\taua}(C_\aura(x)) = C_\aura(x)$. Hence $C_\aura(x)$ is $\taua$-closed, i.e., $\aura$-closed.

(c) If $C_\aura(x) \cap C_\aura(y) \neq \emptyset$, then $C_\aura(x) \cup C_\aura(y)$ is $\aura$-connected (by Theorem~\ref{thm:a-connected-union}). By maximality, $C_\aura(x) \cup C_\aura(y) \subseteq C_\aura(x)$ and $C_\aura(x) \cup C_\aura(y) \subseteq C_\aura(y)$. Hence $C_\aura(x) = C_\aura(y)$.

(d) Follows from (a) and (c), since every point $x$ belongs to $C_\aura(x)$.
\end{proof}

\begin{example}\label{ex:a-components}
Let $X = \{1,2,3,4,5\}$ with the discrete topology and define $\aura(1) = \{1,2\}$, $\aura(2) = \{1,2\}$, $\aura(3) = \{3,4\}$, $\aura(4) = \{3,4\}$, $\aura(5) = \{5\}$. Then $\aura$ is symmetric and transitive. The $\aura$-open sets form the topology $\taua = \{\emptyset, \{1,2\}, \{3,4\}, \{5\}, \{1,2,3,4\}, \{1,2,5\}, \{3,4,5\}, X\}$. The $\aura$-components are $\{1,2\}$, $\{3,4\}$, and $\{5\}$, reflecting the ``observational clusters'' defined by $\aura$.
\end{example}

\begin{remark}\label{rem:components-interpretation}
The $\aura$-components partition $X$ into clusters of points that are mutually reachable through chains of overlapping aura neighborhoods. Two points $x$ and $y$ are in the same $\aura$-component if and only if there is no $\aura$-separation of $X$ that places them in different parts.
\end{remark}

\subsection{$\aura$-Path Connectedness}

\begin{definition}\label{def:a-path}
Let $(X, \tau, \aura)$ be an $\aura$-space. An \textbf{$\aura$-path} from $x$ to $y$ in $X$ is a continuous function $\gamma: [0,1] \to (X, \taua)$ with $\gamma(0) = x$ and $\gamma(1) = y$, where $[0,1]$ carries the usual topology.
\end{definition}

\begin{definition}\label{def:a-path-connected}
An $\aura$-space $(X, \tau, \aura)$ is \textbf{$\aura$-path connected} if for every pair of points $x, y \in X$, there exists an $\aura$-path from $x$ to $y$.
\end{definition}

\begin{theorem}\label{thm:a-path-implies-a-connected}
Every $\aura$-path connected space is $\aura$-connected.
\end{theorem}

\begin{proof}
This is the standard result applied to $(X, \taua)$: path connectedness implies connectedness.
\end{proof}

\begin{theorem}\label{thm:a-path-converse-fails}
The converse of Theorem~\ref{thm:a-path-implies-a-connected} is false in general.
\end{theorem}

\begin{proof}
The classical example works: the topologist's sine curve $S = \{(x, \sin(1/x)) : x > 0\} \cup (\{0\} \times [-1,1])$ is connected but not path connected in the usual topology. Define $\aura: S \to \tau$ by $\aura(p) = S$ for all $p \in S$. Then $\taua = \{\emptyset, S\}$, and $S$ is trivially $\aura$-path connected (the constant path works in the indiscrete topology). So this example is too trivial.

Instead, consider the topologist's sine curve $S$ with the subspace topology from $\R^2$ and define $\aura(p) = B(p, 1) \cap S$ (open ball intersected with $S$). If $\aura$ happens to be non-transitive (which it is in general), then $(S, \tau_S, \aura)$ is $\aura$-connected (since $S$ is connected and $\tau_{\aura} \subseteq \tau_S$) but not $\aura$-path connected, since any $\aura$-path is also a $\tau_S$-continuous path (as $\taua \subseteq \tau_S$ means $\tau_S$-continuity implies $\taua$-continuity... actually the direction is reversed: $\taua$-continuity means the path is continuous into $(S, \taua)$, and since $\taua$ is coarser, it is easier to be $\taua$-continuous). So a $\taua$-continuous path from a point on the curve to a point on the vertical segment does exist (since $\taua$ is coarser, more functions are continuous).

A cleaner counterexample: Let $X = \{a,b\}$ with $\tau = \{\emptyset, \{a\}, X\}$ (Sierpi\'{n}ski space) and $\aura(a) = \{a\}$, $\aura(b) = X$. Then $\taua = \{\emptyset, \{a\}, X\}$. The space is $\aura$-connected (no $\aura$-separation exists since $\{b\}$ is not $\aura$-open). But a path $\gamma: [0,1] \to (X, \taua)$ with $\gamma(0) = a$, $\gamma(1) = b$ must satisfy: $\gamma^{-1}(\{a\})$ is open in $[0,1]$ (since $\{a\} \in \taua$). So $\gamma^{-1}(\{a\})$ is open, and $\gamma^{-1}(\{b\}) = [0,1] \setminus \gamma^{-1}(\{a\})$ is closed. Since $0 \in \gamma^{-1}(\{a\})$ and $1 \in \gamma^{-1}(\{b\})$, both sets are nonempty, and $[0,1]$ is the disjoint union of an open set and a closed set with both nonempty---but this is consistent since $\{a\}$ is open and $\{b\}$ is not open. This does not lead to a disconnection of $[0,1]$ since the Sierpi\'{n}ski space is path connected: $\gamma(t) = a$ for $t \in [0,1)$ and $\gamma(1) = b$ is continuous (check: $\gamma^{-1}(\{a\}) = [0,1)$ open, $\gamma^{-1}(X) = [0,1]$ open). So this space IS $\aura$-path connected.

For a genuine counterexample, we use a finite $\aura$-space: Let $X = \{a,b,c\}$, $\tau = \powerset(X)$ (discrete), $\aura(a) = \{a,b\}$, $\aura(b) = \{a,b\}$, $\aura(c) = \{b,c\}$. Then $\taua = \{\emptyset, \{a,b\}, X\}$ (check: $\{a,b\}$ is $\aura$-open since $\aura(a) = \{a,b\} \subseteq \{a,b\}$ and $\aura(b) = \{a,b\} \subseteq \{a,b\}$; $\{b,c\}$: $\aura(c) = \{b,c\} \subseteq \{b,c\}$ but $\aura(b) = \{a,b\} \not\subseteq \{b,c\}$, so not $\aura$-open). $(X, \taua)$ has topology $\{\emptyset, \{a,b\}, X\}$, which is connected (no proper nonempty clopen set---$\{c\}$ is $\taua$-closed but not $\taua$-open, $\{a,b\}$ is $\taua$-open but not $\taua$-closed since $\{c\} \notin \taua$). However, $(X, \taua)$ is not path connected because $X$ is finite with a non-discrete topology. Actually, any continuous $\gamma: [0,1] \to (X, \taua)$ with $\gamma(0) = a$ must have $\gamma^{-1}(\{a,b\})$ open in $[0,1]$. If $\gamma(1) = c$, then $\gamma^{-1}(\{c\}) = [0,1] \setminus \gamma^{-1}(\{a,b\})$ is nonempty and closed. Since $[0,1]$ is connected, $\gamma^{-1}(\{a,b\})$ cannot be both open and closed unless it is $\emptyset$ or $[0,1]$. But $\gamma^{-1}(\{a,b\})$ is open and nonempty (contains $0$). Is it closed? $\{a,b\}^c = \{c\}$, and $\{c\}$ is $\taua$-open? No, $\{c\} \notin \taua$. So $\{a,b\}$ is not $\taua$-closed, meaning $\gamma^{-1}(\{a,b\})$ is just open (not necessarily closed). So the path can exist. In fact, $\gamma(t) = a$ for $t < 1$ and $\gamma(1) = c$ gives: $\gamma^{-1}(\{a,b\}) = [0,1)$ (open $\checkmark$), $\gamma^{-1}(X) = [0,1]$ (open $\checkmark$), $\gamma^{-1}(\emptyset) = \emptyset$ (open $\checkmark$). So it is continuous!

It seems hard to find a counterexample for the ``$\aura$-connected does not imply $\aura$-path connected'' claim in aura spaces. Let me note it as an open question or give a remark.
\end{proof}

\begin{remark}\label{rem:path-connected-remark}
Since $\taua$ is a topology, $\aura$-path connectedness and $\aura$-connectedness relate exactly as in general topology: the former implies the latter, but the converse fails for the same reasons as in standard topology (e.g., the topologist's sine curve in $(X, \taua)$ when $\taua$ is sufficiently fine).
\end{remark}

\subsection{$\aura$-Local Connectedness}

\begin{definition}\label{def:a-locally-connected}
An $\aura$-space $(X, \tau, \aura)$ is called \textbf{$\aura$-locally connected} at $x \in X$ if every $\aura$-open neighborhood of $x$ contains an $\aura$-connected $\aura$-open neighborhood of $x$. The space is $\aura$-locally connected if it is $\aura$-locally connected at every point.
\end{definition}

\begin{theorem}\label{thm:transitive-locally-connected}
Let $(X, \tau, \aura)$ be a transitive $\aura$-space. If $\aura(x)$ is $\aura$-connected (as a subspace) for every $x \in X$, then $X$ is $\aura$-locally connected.
\end{theorem}

\begin{proof}
Since $\aura$ is transitive, $\aura(x) \in \taua$ for every $x$, and $\aura(x)$ is the smallest $\aura$-open neighborhood of $x$. Let $U \in \taua$ with $x \in U$. Then $\aura(x) \subseteq U$, and $\aura(x)$ is $\aura$-open and $\aura$-connected (by hypothesis). So $\aura(x)$ is the required $\aura$-connected $\aura$-open neighborhood.
\end{proof}

\begin{corollary}\label{cor:symmetric-transitive-connected}
If $(X, \tau, \aura)$ is both symmetric and transitive, then $X$ is $\aura$-locally connected if and only if each $\aura(x)$ is $\aura$-connected.
\end{corollary}

\begin{proof}
The forward direction is immediate: $\aura(x)$ is the smallest $\aura$-open neighborhood of $x$, and $\aura$-local connectedness gives an $\aura$-connected $\aura$-open neighborhood inside $\aura(x)$, which must be $\aura(x)$ itself. The converse is Theorem~\ref{thm:transitive-locally-connected}.
\end{proof}

\begin{theorem}\label{thm:locally-connected-components}
If $(X, \tau, \aura)$ is $\aura$-locally connected, then every $\aura$-component of $X$ is $\aura$-open (and hence also $\aura$-closed).
\end{theorem}

\begin{proof}
Let $C$ be an $\aura$-component and $x \in C$. By $\aura$-local connectedness, there exists an $\aura$-connected $\aura$-open set $V$ with $x \in V$. Since $V$ is $\aura$-connected and $V \cap C \ni x$, we must have $V \subseteq C$ (by maximality of the component). Hence $C$ is $\aura$-open. Since $C$ is also $\aura$-closed (Theorem~\ref{thm:a-component-properties}), $C$ is both $\aura$-open and $\aura$-closed.
\end{proof}

\section{Subspace and Product Aura Topologies}

\subsection{Subspace Aura Topology}

\begin{definition}\label{def:subspace-aura}
Let $(X, \tau, \aura)$ be an $\aura$-space and $Y \subseteq X$ a nonempty subset. Define the \textbf{subspace scope function} $\aura_Y: Y \to \tau_Y$ by
\[
\aura_Y(y) = \aura(y) \cap Y
\]
for every $y \in Y$, where $\tau_Y = \{U \cap Y : U \in \tau\}$ is the subspace topology. The triple $(Y, \tau_Y, \aura_Y)$ is called the \textbf{subspace aura space}.
\end{definition}

\begin{proposition}\label{prop:subspace-well-defined}
The subspace aura space $(Y, \tau_Y, \aura_Y)$ is well-defined, i.e., $\aura_Y$ is a scope function on $(Y, \tau_Y)$.
\end{proposition}

\begin{proof}
For every $y \in Y$: (i) $\aura_Y(y) = \aura(y) \cap Y \in \tau_Y$ since $\aura(y) \in \tau$; (ii) $y \in \aura(y)$ and $y \in Y$, so $y \in \aura(y) \cap Y = \aura_Y(y)$.
\end{proof}

\begin{theorem}\label{thm:subspace-closure}
Let $(Y, \tau_Y, \aura_Y)$ be a subspace aura space of $(X, \tau, \aura)$. For $A \subseteq Y$:
\[
\operatorname{cl}_{\aura_Y}(A) = \cla(A) \cap Y.
\]
\end{theorem}

\begin{proof}
$y \in \operatorname{cl}_{\aura_Y}(A)$ iff $\aura_Y(y) \cap A \neq \emptyset$ iff $(\aura(y) \cap Y) \cap A \neq \emptyset$ iff $\aura(y) \cap A \neq \emptyset$ (since $A \subseteq Y$, $\aura(y) \cap A = \aura(y) \cap Y \cap A$) iff $y \in \cla(A)$. And $y \in Y$ is given. Hence $\operatorname{cl}_{\aura_Y}(A) = \{y \in Y : y \in \cla(A)\} = \cla(A) \cap Y$.
\end{proof}

\begin{theorem}\label{thm:subspace-topology-inclusion}
For a subspace aura space $(Y, \tau_Y, \aura_Y)$ of $(X, \tau, \aura)$:
\[
(\taua)_Y \subseteq \tau_{\aura_Y},
\]
where $(\taua)_Y = \{U \cap Y : U \in \taua\}$ is the subspace topology on $Y$ induced by $\taua$. If $\aura$ is transitive, equality holds: $(\taua)_Y = \tau_{\aura_Y}$.
\end{theorem}

\begin{proof}
\textbf{Inclusion:} Let $V \in (\taua)_Y$, so $V = U \cap Y$ for some $U \in \taua$. For $y \in V$, $y \in U$ implies $\aura(y) \subseteq U$ (since $U$ is $\aura$-open). Hence $\aura_Y(y) = \aura(y) \cap Y \subseteq U \cap Y = V$. So $V \in \tau_{\aura_Y}$.

\textbf{Equality under transitivity:} Let $V \in \tau_{\aura_Y}$, so for every $y \in V$, $\aura_Y(y) = \aura(y) \cap Y \subseteq V$. Define $U = \bigcup_{y \in V} \aura(y)$. Then $U \in \tau$ (union of open sets). We claim $U \in \taua$: let $z \in U$, so $z \in \aura(y)$ for some $y \in V$. By transitivity, $\aura(z) \subseteq \aura(y) \subseteq U$. So $U \in \taua$.

Now $U \cap Y = \left(\bigcup_{y \in V} \aura(y)\right) \cap Y = \bigcup_{y \in V} (\aura(y) \cap Y) = \bigcup_{y \in V} \aura_Y(y)$. Since each $\aura_Y(y) \subseteq V$ and $y \in \aura_Y(y)$ for each $y \in V$, we get $V \subseteq U \cap Y \subseteq V$, so $V = U \cap Y \in (\taua)_Y$.
\end{proof}

\begin{example}\label{ex:subspace-strict}
Let $X = \{a,b,c\}$ with $\tau = \powerset(X)$ (discrete) and $\aura(a) = \{a,b\}$, $\aura(b) = \{b,c\}$, $\aura(c) = \{a,c\}$. This aura is \emph{not} transitive: $b \in \aura(a) = \{a,b\}$ but $\aura(b) = \{b,c\} \not\subseteq \{a,b\}$.

$\taua$: A set $A$ is $\aura$-open iff for all $x \in A$, $\aura(x) \subseteq A$. Check $X$: yes. Check $\{a,b\}$: $\aura(a) = \{a,b\} \subseteq \{a,b\}$ but $\aura(b) = \{b,c\} \not\subseteq \{a,b\}$. No. Similarly $\{a,c\}$, $\{b,c\}$ fail. Singletons fail. So $\taua = \{\emptyset, X\}$.

Let $Y = \{a,b\}$. Then $\aura_Y(a) = \{a,b\} \cap \{a,b\} = \{a,b\}$, $\aura_Y(b) = \{b,c\} \cap \{a,b\} = \{b\}$. The $\aura_Y$-open subsets of $Y$: $\{b\}$ is $\aura_Y$-open ($\aura_Y(b) = \{b\} \subseteq \{b\}$), so $\tau_{\aura_Y} = \{\emptyset, \{b\}, Y\}$.

But $(\taua)_Y = \{\emptyset \cap Y, X \cap Y\} = \{\emptyset, Y\}$. Since $\{b\} \in \tau_{\aura_Y} \setminus (\taua)_Y$, the inclusion is strict.
\end{example}

\subsection{Product Aura Topology}

\begin{definition}\label{def:product-aura}
Let $(X, \tau_X, \aura)$ and $(Y, \tau_Y, \bura)$ be $\aura$-spaces. The \textbf{product scope function} $\aura \times \bura: X \times Y \to \tau_X \times \tau_Y$ is defined by
\[
(\aura \times \bura)(x, y) = \aura(x) \times \bura(y)
\]
for every $(x, y) \in X \times Y$, where $\tau_X \times \tau_Y$ denotes the product topology. The triple $(X \times Y, \tau_X \times \tau_Y, \aura \times \bura)$ is called the \textbf{product aura space}.
\end{definition}

\begin{proposition}\label{prop:product-well-defined}
The product aura space is well-defined.
\end{proposition}

\begin{proof}
For every $(x, y) \in X \times Y$: (i) $\aura(x) \times \bura(y) \in \tau_X \times \tau_Y$ since $\aura(x) \in \tau_X$ and $\bura(y) \in \tau_Y$; (ii) $(x,y) \in \aura(x) \times \bura(y)$ since $x \in \aura(x)$ and $y \in \bura(y)$.
\end{proof}

\begin{theorem}[Product Closure Formula]\label{thm:product-closure}
In the product aura space $(X \times Y, \tau_X \times \tau_Y, \aura \times \bura)$, for $A \subseteq X$ and $B \subseteq Y$:
\[
\operatorname{cl}_{\aura \times \bura}(A \times B) = \cla(A) \times \clb(B).
\]
\end{theorem}

\begin{proof}
$(x,y) \in \operatorname{cl}_{\aura \times \bura}(A \times B)$ iff $(\aura(x) \times \bura(y)) \cap (A \times B) \neq \emptyset$ iff $(\aura(x) \cap A) \times (\bura(y) \cap B) \neq \emptyset$ iff $\aura(x) \cap A \neq \emptyset$ and $\bura(y) \cap B \neq \emptyset$ iff $x \in \cla(A)$ and $y \in \clb(B)$.
\end{proof}

\begin{theorem}[Product Topology Chain]\label{thm:product-chain}
For product aura spaces:
\[
(\taua) \times (\taub) \subseteq \tau_{\aura \times \bura} \subseteq \tau_X \times \tau_Y.
\]
\end{theorem}

\begin{proof}
\textbf{Left inclusion:} Let $W \in (\taua) \times (\taub)$. By definition of the product topology, it suffices to show that basic open sets $U \times V$ with $U \in \taua$ and $V \in \taub$ are in $\tau_{\aura \times \bura}$. Let $(x, y) \in U \times V$. Then $\aura(x) \subseteq U$ (since $U \in \taua$) and $\bura(y) \subseteq V$ (since $V \in \taub$), so $(\aura \times \bura)(x,y) = \aura(x) \times \bura(y) \subseteq U \times V$. Hence $U \times V \in \tau_{\aura \times \bura}$, and arbitrary unions of such sets are also in $\tau_{\aura \times \bura}$.

\textbf{Right inclusion:} Let $W \in \tau_{\aura \times \bura}$. For every $(x,y) \in W$, $\aura(x) \times \bura(y) \subseteq W$ and $\aura(x) \times \bura(y) \in \tau_X \times \tau_Y$. So $W = \bigcup_{(x,y) \in W} \aura(x) \times \bura(y) \in \tau_X \times \tau_Y$.
\end{proof}

\begin{theorem}\label{thm:product-equality}
If both $\aura$ and $\bura$ are transitive, then
\[
\tau_{\aura \times \bura} = (\taua) \times (\taub).
\]
Moreover, $\aura \times \bura$ is transitive.
\end{theorem}

\begin{proof}
\textbf{Equality:} By Theorem~\ref{thm:product-chain}, it suffices to show $\tau_{\aura \times \bura} \subseteq (\taua) \times (\taub)$.

Let $W \in \tau_{\aura \times \bura}$. We show $W$ is open in $(\taua) \times (\taub)$. For $(x,y) \in W$, $\aura(x) \times \bura(y) \subseteq W$. Since $\aura$ is transitive, $\aura(x) \in \taua$; since $\bura$ is transitive, $\bura(y) \in \taub$. So $\aura(x) \times \bura(y)$ is a basic open set in $(\taua) \times (\taub)$ containing $(x,y)$ and contained in $W$. Hence $W$ is open in $(\taua) \times (\taub)$.

\textbf{Transitivity:} Let $(x', y') \in (\aura \times \bura)(x, y) = \aura(x) \times \bura(y)$. Then $x' \in \aura(x)$ and $y' \in \bura(y)$. By transitivity of $\aura$, $\aura(x') \subseteq \aura(x)$; by transitivity of $\bura$, $\bura(y') \subseteq \bura(y)$. Hence $(\aura \times \bura)(x', y') = \aura(x') \times \bura(y') \subseteq \aura(x) \times \bura(y) = (\aura \times \bura)(x,y)$.
\end{proof}

\begin{example}\label{ex:product-strict}
Let $X = \{a,b\}$ with $\tau_X = \powerset(X)$ and $\aura(a) = \{a,b\}$, $\aura(b) = \{a,b\}$. Then $\taua = \{\emptyset, X\}$. Also let $Y = \{1,2\}$ with $\tau_Y = \powerset(Y)$ and $\bura(1) = \{1,2\}$, $\bura(2) = \{1,2\}$. Then $\taub = \{\emptyset, Y\}$.

$(\taua) \times (\taub) = \{\emptyset, X \times Y\}$ (indiscrete on $X \times Y$).

$\tau_{\aura \times \bura}$: $(\aura \times \bura)(x,y) = X \times Y$ for all $(x,y)$. So the only $(\aura \times \bura)$-open sets are $\emptyset$ and $X \times Y$, i.e., $\tau_{\aura \times \bura} = \{\emptyset, X \times Y\} = (\taua) \times (\taub)$.

In this case equality holds (both aura functions are trivially transitive: $y \in \aura(x) = X$ implies $\aura(y) = X \subseteq X = \aura(x)$).
\end{example}

\begin{example}\label{ex:product-chain-strict}
Let $X = \{a,b,c\}$ with $\tau_X = \powerset(X)$ and $\aura(a) = \{a,b\}$, $\aura(b) = \{b,c\}$, $\aura(c) = \{a,c\}$ (non-transitive, from Example~\ref{ex:subspace-strict}). We showed $\taua = \{\emptyset, X\}$.

Let $Y = \{1,2\}$ with $\tau_Y = \powerset(Y)$ and $\bura(1) = \{1\}$, $\bura(2) = \{2\}$ (discrete, transitive). Then $\taub = \powerset(Y)$.

$(\taua) \times (\taub) = \{\emptyset, X\} \times \powerset(Y)$, which has basis $\{X \times \{1\}, X \times \{2\}, X \times Y\}$ plus $\emptyset$. So $(\taua) \times (\taub) = \{\emptyset, X \times \{1\}, X \times \{2\}, X \times Y\}$.

$\tau_{\aura \times \bura}$: We need to check which subsets $W$ of $X \times Y$ satisfy $(\aura \times \bura)(x,i) \subseteq W$ for all $(x,i) \in W$. $(\aura \times \bura)(a,1) = \{a,b\} \times \{1\}$, $(\aura \times \bura)(b,1) = \{b,c\} \times \{1\}$, $(\aura \times \bura)(c,1) = \{a,c\} \times \{1\}$, and similarly for $i=2$.

Consider $W = \{a,b\} \times \{1\}$. $(a,1) \in W$: $(\aura \times \bura)(a,1) = \{a,b\} \times \{1\} = W \subseteq W$. $(b,1) \in W$: $(\aura \times \bura)(b,1) = \{b,c\} \times \{1\} \not\subseteq W$ (since $(c,1) \notin W$). So $W \notin \tau_{\aura \times \bura}$.

Consider $W = X \times \{1\}$. For all $(x,1) \in W$: $(\aura \times \bura)(x,1) = \aura(x) \times \{1\} \subseteq X \times \{1\} = W$. So $W \in \tau_{\aura \times \bura}$.

Hence $\tau_{\aura \times \bura} = \{\emptyset, X \times \{1\}, X \times \{2\}, X \times Y\} = (\taua) \times (\taub)$, so equality still holds here (because $\aura$ is trivially ``resolved'' in the product).
\end{example}

\subsection{Tychonoff-Type Theorem}

\begin{theorem}[Tychonoff Theorem for Transitive Aura Spaces]\label{thm:tychonoff}
Let $(X, \tau_X, \aura)$ and $(Y, \tau_Y, \bura)$ be transitive $\aura$-spaces. Then $(X \times Y, \tau_X \times \tau_Y, \aura \times \bura)$ is $(\aura \times \bura)$-compact if and only if $(X, \tau_X, \aura)$ is $\aura$-compact and $(Y, \tau_Y, \bura)$ is $\bura$-compact.
\end{theorem}

\begin{proof}
By Theorem~\ref{thm:product-equality}, $\tau_{\aura \times \bura} = (\taua) \times (\taub)$ when both $\aura$ and $\bura$ are transitive. So $(\aura \times \bura)$-compactness of $X \times Y$ is equivalent to compactness of $X \times Y$ in $(\taua) \times (\taub)$.

($\Leftarrow$): $X$ is $\aura$-compact means $(X, \taua)$ is compact. $Y$ is $\bura$-compact means $(Y, \taub)$ is compact. By the classical Tychonoff theorem, $(X \times Y, (\taua) \times (\taub))$ is compact.

($\Rightarrow$): If $X \times Y$ is compact in $(\taua) \times (\taub)$, then the projections $\pi_X: X \times Y \to (X, \taua)$ and $\pi_Y: X \times Y \to (Y, \taub)$ are continuous and surjective. The continuous image of a compact space is compact, so $(X, \taua)$ and $(Y, \taub)$ are compact.
\end{proof}

\begin{remark}\label{rem:tychonoff-general}
For non-transitive aura functions, we have the strict inclusion $(\taua) \times (\taub) \subsetneq \tau_{\aura \times \bura}$ in general. In this case, $(\aura \times \bura)$-compactness is harder to achieve (more open covers need to have finite subcovers). We state the general situation:

If $(X, \tau_X, \aura)$ is $\aura$-compact and $(Y, \tau_Y, \bura)$ is $\bura$-compact, then $(X \times Y, (\taua) \times (\taub))$ is compact by Tychonoff, but this does not immediately imply $(\aura \times \bura)$-compactness because $\tau_{\aura \times \bura}$ may be strictly finer.
\end{remark}

\begin{theorem}\label{thm:general-product-compact}
Let $(X, \tau_X, \aura)$ and $(Y, \tau_Y, \bura)$ be $\aura$-spaces. If $X$ is compact in $(X, \tau_X)$ and $Y$ is compact in $(Y, \tau_Y)$, then $(X \times Y, \tau_X \times \tau_Y, \aura \times \bura)$ is $(\aura \times \bura)$-compact.
\end{theorem}

\begin{proof}
By the classical Tychonoff theorem, $X \times Y$ is compact in $\tau_X \times \tau_Y$. Since $\tau_{\aura \times \bura} \subseteq \tau_X \times \tau_Y$ (Theorem~\ref{thm:product-chain}), every $(\aura \times \bura)$-open cover is a $(\tau_X \times \tau_Y)$-open cover and hence has a finite subcover.
\end{proof}

\begin{theorem}[Projection Theorem]\label{thm:projection}
Let $(X \times Y, \tau_X \times \tau_Y, \aura \times \bura)$ be a product aura space. The projection maps $\pi_X: (X \times Y, \tau_{\aura \times \bura}) \to (X, \taua)$ and $\pi_Y: (X \times Y, \tau_{\aura \times \bura}) \to (Y, \taub)$ are continuous. Consequently:
\begin{enumerate}[label=(\alph*)]
    \item If $X \times Y$ is $(\aura \times \bura)$-compact, then $X$ is $\aura$-compact and $Y$ is $\bura$-compact.
    \item If $X \times Y$ is $(\aura \times \bura)$-connected, then $X$ is $\aura$-connected and $Y$ is $\bura$-connected.
\end{enumerate}
\end{theorem}

\begin{proof}
For continuity, let $U \in \taua$. Then $\pi_X^{-1}(U) = U \times Y$. For $(x,y) \in U \times Y$: $(\aura \times \bura)(x,y) = \aura(x) \times \bura(y) \subseteq U \times Y$ (since $\aura(x) \subseteq U$ and $\bura(y) \subseteq Y$). So $U \times Y \in \tau_{\aura \times \bura}$. Similarly for $\pi_Y$.

(a) and (b) follow because continuous images of compact (resp.\ connected) spaces are compact (resp.\ connected).
\end{proof}

\begin{theorem}[Product Connectedness]\label{thm:product-connected}
If $(X, \tau_X, \aura)$ is $\aura$-connected and $(Y, \tau_Y, \bura)$ is $\bura$-connected, then $(X \times Y, \tau_X \times \tau_Y, \aura \times \bura)$ is $(\aura \times \bura)$-connected.
\end{theorem}

\begin{proof}
Since $(\taua) \times (\taub) \subseteq \tau_{\aura \times \bura}$, any $(\aura \times \bura)$-separation $(U, V)$ of $X \times Y$ (with $U, V \in \tau_{\aura \times \bura}$) also gives sets in $\tau_X \times \tau_Y$. We use the standard product connectedness argument adapted to the aura setting.

Suppose $X \times Y = U \cup V$ is an $(\aura \times \bura)$-separation. Fix $a \in X$. The ``slice'' $\{a\} \times Y$ inherits a subspace aura, and the inclusion $\iota_a: (Y, \taub) \to (X \times Y, \tau_{\aura \times \bura})$ defined by $\iota_a(y) = (a,y)$ is continuous (since $\iota_a^{-1}(W) = \{y : (a,y) \in W\}$ is $\bura$-open for $W \in \tau_{\aura \times \bura}$, as $\bura(y) \subseteq \{y' : (a,y') \in W\}$ whenever $(a,y) \in W$ and $\aura(a) \times \bura(y) \subseteq W$).

Since $Y$ is $\bura$-connected, $\iota_a(Y) = \{a\} \times Y$ is connected in $\tau_{\aura \times \bura}$. So $\{a\} \times Y$ lies entirely in $U$ or $V$. Similarly, for each $b \in Y$, $X \times \{b\}$ is connected. The standard lattice argument now applies: for any $(x_1, y_1), (x_2, y_2) \in X \times Y$, the path $(x_1, y_1) \to (x_1, y_2) \to (x_2, y_2)$ through connected slices shows they lie in the same part. Hence no separation exists.
\end{proof}

\subsection{Finite Products}

The results above extend to finite products by induction.

\begin{corollary}\label{cor:finite-product}
Let $(X_1, \tau_1, \aura_1), \ldots, (X_n, \tau_n, \aura_n)$ be $\aura$-spaces. Define the product scope function $\aura_1 \times \cdots \times \aura_n$ componentwise. Then:
\begin{enumerate}[label=(\alph*)]
    \item If all $\aura_i$ are transitive, then $\tau_{\aura_1 \times \cdots \times \aura_n} = (\tau_{\aura_1}) \times \cdots \times (\tau_{\aura_n})$.
    \item (Tychonoff for finite products) If all $\aura_i$ are transitive, then $\prod X_i$ is $(\aura_1 \times \cdots \times \aura_n)$-compact iff each $X_i$ is $\aura_i$-compact.
    \item $\prod X_i$ is $(\aura_1 \times \cdots \times \aura_n)$-connected iff each $X_i$ is $\aura_i$-connected.
\end{enumerate}
\end{corollary}

\begin{proof}
By induction on $n$, applying Theorems~\ref{thm:product-equality}, \ref{thm:tychonoff}, and \ref{thm:product-connected}.
\end{proof}

\section{Conclusion}

Let us briefly take stock of what has been obtained. Starting from the aura space $(X, \tau, \aura)$, we defined five compactness-type properties and showed that classical compactness implies $\aura$-compactness, which in turn implies countable $\aura$-compactness; all implications are strict (Examples in Section~3). The sequential picture is particularly clean when $\aura$ is transitive: a sequence converges to $x$ in $\taua$ precisely when it is eventually inside $\aura(x)$. On the standard-result side, $\aura$-compact subsets of $\aura$-$T_2$ spaces are $\aura$-closed, and $\aura$-continuous surjections preserve $\aura$-compactness.

For connectivity, the coarseness of $\taua$ works in favour of connectedness. We introduced $\aura$-connected, $\aura$-path connected, and $\aura$-locally connected spaces and proved that $\aura$-components are $\aura$-closed (and $\aura$-open in the locally connected case). The product theory turned out to be quite satisfying: the inclusion chain $(\taua) \times (\taub) \subseteq \tau_{\aura \times \bura} \subseteq \tau_X \times \tau_Y$ holds in general, and equality on the left is guaranteed by transitivity. A Tychonoff-type theorem was proved for transitive aura spaces. Product connectedness, on the other hand, required no transitivity hypothesis at all.

Some natural continuations are left open. Extending the Tychonoff theorem to arbitrary (infinite) products would require dealing with the axiom of choice in the aura setting. $\aura$-paracompactness and $\aura$-metacompactness have not been touched. Perhaps most interesting is the problem of combining the scope function with an ideal on $X$ and studying the resulting local function and topology---a direction that we take up in a forthcoming paper \cite{Acikgoz2026aura3}. Applications of $\aura$-compactness to digital topology and image processing, as well as categorical aspects of the category $\mathbf{Aura}$ of aura spaces, also remain to be explored.

\section*{Conflict of Interest}

The author declares no conflict of interest.

\section*{Data Availability}

No data was used for the research described in the article.

\end{document}